\documentclass[12pt,reqno]{amsart}
\usepackage{amsmath}
\usepackage{amssymb}
\usepackage[left=3cm,top=3cm,right=3cm,bottom=3cm]{geometry}
\usepackage{epsfig}

\begin{document}
\newtheorem{theorem}{Theorem}[section]
\newtheorem{lemma}[theorem]{Lemma}
\newtheorem{definition}[theorem]{Definition}
\newtheorem{conjecture}[theorem]{Conjecture}
\newtheorem{proposition}[theorem]{Proposition}
\newtheorem{algorithm}[theorem]{Algorithm}
\newtheorem{corollary}[theorem]{Corollary}
\newtheorem{observation}[theorem]{Observation}
\newtheorem{problem}[theorem]{Open Problem}
\newcommand{\noin}{\noindent}
\newcommand{\ind}{\indent}
\newcommand{\al}{\alpha}
\newcommand{\om}{\omega}
\newcommand{\pp}{\mathcal P}
\newcommand{\ppp}{\mathfrak P}
\newcommand{\R}{{\mathbb R}}
\newcommand{\N}{{\mathbb N}}
\newcommand\eps{\varepsilon}
\newcommand{\E}{\mathbb E}
\newcommand{\Prob}{\mathbb{P}}
\newcommand{\pl}{\textrm{C}}
\newcommand{\dang}{\textrm{dang}}
\renewcommand{\labelenumi}{(\roman{enumi})}
\newcommand{\bc}{\bar c}
\newcommand{\G}{{\mathcal{G}}}
\newcommand{\expect}[1]{\E \left [ #1 \right ]}
\newcommand{\ceil}[1]{\left \lceil #1 \right \rceil}
\newcommand{\of}[1]{\left( #1 \right)}
\newcommand{\set}[1]{\left\{ #1 \right\}}
\newcommand{\size}[1]{\left \vert #1 \right \vert}
\newcommand{\floor}[1]{\left \lfloor #1 \right \rfloor}

\title{Containment game played on random graphs: another zig-zag theorem}

\author{Pawe{\l} Pra{\l}at}
\address{Department of Mathematics, Ryerson University, Toronto, ON, Canada, M5B 2K3}
\email{\tt pralat@ryerson.ca}

\thanks{Supported by grants from NSERC and Ryerson}

\maketitle

\begin{abstract}
We consider a variant of the game of Cops and Robbers, called Containment, in which cops move from edge to adjacent edge, the robber moves from vertex to adjacent vertex (but cannot move along an edge occupied by a cop). The cops win by ``containing'' the robber, that is, by occupying all edges incident with a vertex occupied by the robber. The minimum number of cops, $\xi(G)$, required to contain a robber played on a graph $G$ is called the containability number, a natural counterpart of the well-known cop number $c(G)$. This variant of the game was recently introduced by Komarov and Mackey, who proved that for every graph $G$, $c(G) \le \xi(G) \le \gamma(G) \Delta(G)$, where $\gamma(G)$ and $\Delta(G)$ are the domination number and the maximum degree of $G$, respectively. They conjecture that an upper bound can be improved and, in fact, $\xi(G) \le c(G) \Delta(G)$. (Observe that, trivially, $c(G) \le \gamma(G)$.) This seems to be the main question for this game at the moment. By investigating expansion properties, we provide asymptotically almost sure bounds on the containability number of binomial random graphs $\G(n,p)$ for a wide range of $p=p(n)$, showing that it forms an intriguing zigzag shape. This result also proves that the conjecture holds for some range of $p$ (or holds up to a constant or an $O(\log n)$ multiplicative factors for some other ranges). 
\end{abstract}

\section{Introduction}

The game of Cops and Robbers (defined, along with all the standard notation, later in this section) is usually studied in the context of the {\em cop number}, the minimum number of cops needed to ensure a winning strategy. The cop number is often challenging to analyze; establishing upper bounds for this parameter is the focus of Meyniel's conjecture that the cop number of a connected $n$-vertex graph is $O(\sqrt{n}).$ For additional background on Cops and Robbers and Meyniel's conjecture, see the book~\cite{bonato}.

A number of variants of Cops and Robbers have been studied. For example, we may allow a cop to capture the robber from a distance $k$, where $k$ is a non-negative integer~\cite{bonato4}, play on edges~\cite{pawel}, allow one or both players to move with different speeds~\cite{NogaAbbas, fkl} or to teleport, allow the robber to capture the cops~\cite{bonato0}, make the robber invisible or drunk~\cite{drunk1,drunk2}, or allow at most one cop to move in any given round~\cite{oo, hypercube, lazy_gnp}. See Chapter~8 of~\cite{bonato} for a non-comprehensive survey of variants of Cops and Robbers.

\bigskip

In this paper, we consider a variant of the game of Cops and Robbers, called \emph{Containment}, introduced recently by Komarov and Mackey~\cite{komarov}. In this version, cops move from edge to adjacent edge, the robber moves as in the classic game, from vertex to adjacent vertex (but cannot move along an edge occupied by a cop). Formally, the game is played on a finite, simple, and undetected graph. There are two players, a set of \emph{cops} and a single \emph{robber}. The game is played over a sequence of discrete time-steps or \emph{turns}, with the cops going first on turn $0$ and then playing on alternate time-steps.  A \emph{round} of the game is a cop move together with the subsequent robber move.  The cops occupy edges and the robber occupies vertices; for simplicity, we often identify the player with the vertex/edge they occupy. When the robber is ready to move in a round, she can move to a neighbouring vertex but cannot move along an edge occupied by a cop, cops can move to an edge that is incident to their current location. Players can always \emph{pass}, that is, remain on their own vertices/edges. Observe that any subset of cops may move in a given round. The cops win if after some finite number of rounds, all edges incident with the robber are occupied by cops. This is called a \emph{capture}. The robber wins if she can evade capture indefinitely. A \emph{winning strategy for the cops} is a set of rules that if followed, result in a win for the cops. A \emph{winning strategy for the robber} is defined analogously.  As stated earlier, the original game of \emph{Cops and Robbers} is defined almost exactly as this one, with the exception that all players occupy vertices.

If we place a cop at each edge, then the cops are guaranteed to win. Therefore, the minimum number of cops required to win in a graph $G$ is a well-defined positive integer, named the \emph{containability number} of the graph $G.$ Following the notation introduced in~\cite{komarov}, we write $\xi(G)$ for the containability number of a graph $G$ and $c(G)$ for the original \emph{cop-number} of $G$.

\bigskip

In~\cite{komarov}, Komarov and Mackey proved that for every graph $G$, 
$$
c(G) \le \xi(G) \le \gamma(G) \Delta(G), 
$$
where $\gamma(G)$ and $\Delta(G)$ are the domination number and the maximum degree of $G$, respectively. It was conjectured that the upper bound can be strengthened and, in fact, the following holds.

\begin{conjecture}[\cite{komarov}]\label{con:komarov}
For every graph $G$, $\xi(G) \le c(G) \Delta(G)$. 
\end{conjecture}

\noindent Observe that, trivially, $c(G) \le \gamma(G)$ so this would imply the previous result. This seems to be the main question for this variant of the game at the moment. By investigating expansion properties, we provide asymptotically almost sure bounds on the containability number of binomial random graphs $\G(n,p)$ for a wide range of $p=p(n)$, proving that the conjecture holds for some ranges of $p$ (or holds up to a constant or an $O(\log n)$ multiplicative factors for some other ranges of $p$). However, before we state the result, let us introduce the probability space we deal with and mention a few results for the classic cop-number that will be needed to examine the conjecture (since the corresponding upper bound is a function of the cop number).

\bigskip

The \emph{random graph} $\G(n,p)$ consists of the probability space $(\Omega, \mathcal{F}, \Prob)$, where $\Omega$ is the set of all graphs with vertex set $\{1,2,\dots,n\}$, $\mathcal{F}$ is the family of all subsets of $\Omega$, and for every $G \in \Omega$,
$$
\Prob(G) = p^{|E(G)|} (1-p)^{{n \choose 2} - |E(G)|} \,.
$$
This space may be viewed as the set of outcomes of ${n \choose 2}$ independent coin flips, one for each pair $(u,v)$ of vertices, where the probability of success (that is, adding edge $uv$) is $p.$ Note that $p=p(n)$ may (and usually does) tend to zero as $n$ tends to infinity. All asymptotics throughout are as $n \rightarrow \infty $ (we emphasize that the notations $o(\cdot)$ and $O(\cdot)$ refer to functions of $n$, not necessarily positive, whose growth is bounded). We say that an event in a probability space holds \emph{asymptotically almost surely} (or \emph{a.a.s.}) if the probability that it holds tends to $1$ as $n$ goes to infinity.

\bigskip

Let us now briefly describe some known results on the (classic) cop-number of $\G(n,p)$. Bonato, Wang, and the author of this paper investigated such games in $\G(n,p)$ random graphs and in generalizations used to model complex networks with power-law degree distributions (see~\cite{bpw}). From their results it follows that if $2 \log n / \sqrt{n} \le p < 1-\eps$ for some $\eps>0$, then a.a.s. we have that
\begin{equation*}%\label{eq:classig_c}
c(\G(n,p))= \Theta(\log n/p),
\end{equation*}
so Meyniel's conjecture holds a.a.s.\ for such $p$. In fact, for $p=n^{-o(1)}$ we have that a.a.s.\  $c(\G(n,p))=(1+o(1)) \log_{1/(1-p)} n$. A simple argument using dominating sets shows that Meyniel's conjecture also holds a.a.s.\  if $p$ tends to 1 as $n$ goes to infinity (see~\cite{p} for this and stronger results). Bollob\'as, Kun and Leader~\cite{bkl} showed that if $p(n) \ge 2.1 \log n /n$, then a.a.s.
$$
\frac{1}{(pn)^2}n^{ 1/2 - 9/(2\log\log (pn))  }  \le c(\G(n,p))\le 160000\sqrt n \log n\,.
$$
From these results, if $np \ge 2.1 \log n$ and either $np=n^{o(1)}$ or $np=n^{1/2+o(1)}$, then a.a.s.\ $c(\G(n,p))= n^{1/2+o(1)}$. Somewhat surprisingly, between these values it was shown by \L{}uczak and the author of this paper~\cite{lp2} that the cop number has more complicated behaviour. It follows that a.a.s.\ $\log_n  c(\G(n,n^{x-1}))$ is asymptotic to the function $f(x)$ shown in Figure~\ref{fig1} (denoted in blue).
\begin{figure}[h]
\begin{center}
\includegraphics[width=3.3in]{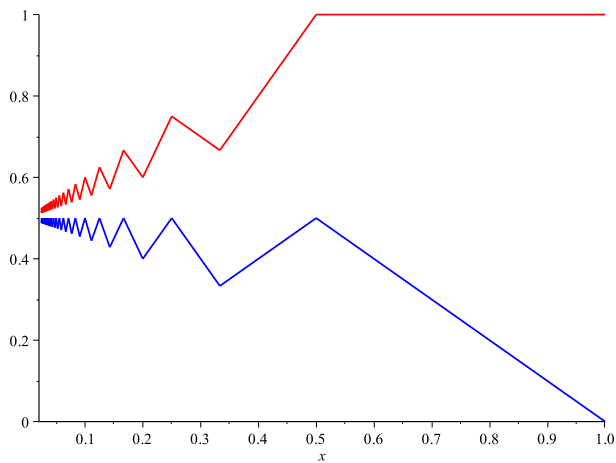}
\end{center}
\caption{The ``zigzag'' functions representing the ordinary cop number (blue) and the containability number (red).}\label{fig1}
\end{figure}

Formally, the following result holds for the classic game.

\begin{theorem}[\cite{lp2, bpw}]\label{thm:zz}
Let $0<\alpha<1$ and $d=d(n)=np=n^{\alpha+o(1)}$.
\begin{enumerate}
\item If $\frac{1}{2j+1}<\alpha<\frac{1}{2j}$ for some integer $j\ge 1$, then a.a.s.\
$$
c(\G(n,p))= \Theta(d^j)\,.
$$
\item If $\frac{1}{2j}<\alpha<\frac{1}{2j-1}$ for some integer $j\ge 2$, then a.a.s.\
\begin{eqnarray*}
c(\G(n,p)) &=& \Omega \left( \frac{n}{d^j} \right), \text{ and } \\
c(\G(n,p)) &=& O \left( \frac{n \log n}{d^j}  \right)\,.
\end{eqnarray*}
\item If $1/2 < \alpha < 1$, then a.a.s.\
$$
c(\G(n,p)) = \Theta \left( \frac {n \log n}{d} \right).
$$
\end{enumerate}
\end{theorem}

The above result shows that Meyniel's conjecture holds a.a.s.\ for random graphs except perhaps when $np=n^{1/(2k)+o(1)}$ for some $k \in \N$, or when $np=n^{o(1)}$. The author of this paper and Wormald showed recently that the conjecture holds a.a.s.\ in $\G(n,p)$~\cite{PW_gnp} as well as in random $d$-regular graphs~\cite{PW_gnd}.

\bigskip

Finally, we are able to state the result of this paper.

\begin{theorem}\label{thm:main}
Let $0<\alpha<1$ and $d=d(n)=np=n^{\alpha+o(1)}$.
\begin{enumerate}
\item If $\frac{1}{2j+1}<\alpha<\frac{1}{2j}$ for some integer $j\ge 1$, then a.a.s.\
$$
\xi(\G(n,p))= \Theta(d^{j+1}) = \Theta(c(\G(n,p)) \cdot \Delta( \G(n,p) ) )\,.
$$
Hence, a.a.s.\ Conjecture~\ref{con:komarov} holds (up to a multiplicative constant factor).
\item If $\frac{1}{2j}<\alpha<\frac{1}{2j-1}$ for some integer $j\ge 2$, then a.a.s.\
\begin{eqnarray*}
\xi(\G(n,p)) &=& \Omega \left( \frac{n}{d^{j-1}} \right), \text{ and } \\
\xi(\G(n,p)) &=& O \left( \frac{n \log n}{d^{j-1}} \right) = O(c(\G(n,p)) \cdot \Delta( \G(n,p) ) \cdot \log n )\,.
\end{eqnarray*}
Hence, a.a.s.\ Conjecture~\ref{con:komarov} holds (up to a multiplicative $O(\log n)$ factor).
\item If $1/2 < \alpha < 1$, then a.a.s.\
$$
\xi(\G(n,p)) = \Theta(n) = \Theta (c(\G(n,p)) \cdot \Delta( \G(n,p) ) / \log n ) \le c(\G(n,p)) \cdot \Delta( \G(n,p).
$$
Hence, a.a.s.\ Conjecture~\ref{con:komarov} holds.
\end{enumerate}
\end{theorem}

It follows that a.a.s.\ $\log_n  \xi(\G(n,n^{x-1}))$ is asymptotic to the function $g(x)$ shown in Figure~\ref{fig1} (denoted in red). The fact the conjecture holds is associated with the observation that $g(x) - f(x) = x$, which is equivalent to saying that a.a.s.\ the ratio $\xi(\G(n,p)) / c(\G(n,p)) = d n^{o(1)} = \Delta(\G(n,p)) \cdot n^{o(1)}$. Moreover, let us mention that Theorem~\ref{thm:main} implies that the conjecture is best possible (again, up to a constant or an $O(\log n)$ multiplicative factors for corresponding ranges of $p$).

\bigskip

Note that in the above result we skip the case when $np=n^{1/k+o(1)}$ for some positive integer $k$ or $np=n^{o(1)}$. It is done for a technical reason: an argument for the lower bound for $\xi(\G(n,p))$ uses a technical lemma from~\cite{lp2} that, in turn, uses Corollary 2.6 from~\cite{Vu} which is stated only for $np=n^{\alpha+o(1)}$, where $\alpha \neq 1/k$ for any positive integer $k$. Clearly, one can repeat the argument given in~\cite{Vu}, which is a very nice but slightly technical application of the polynomial concentration method inequality by Kim and Vu. However, in order to make the paper easier and more compact, a ready-to-use lemma from~\cite{lp2} is used and we concentrate on the ``linear'' parts of the graph of the zigzag function.  Nonetheless, similarly to the corresponding result for $c(\G(n,p))$, one can expect that, up to a  factor of $\log^{O(1)}n$,  the result extends naturally also to the case  $np=n^{1/k+o(1)}$ as well.  

On the other hand, there is no problem with the upper bound so the case when $np=n^{1/k+o(1)}$ for some positive integer $k$ is also investigated (see below for a precise statement). Moreover, some expansion properties that were used to prove that Meyniel's conjecture holds for $\G(n,p)$~\cite{PW_gnp} are incorporated here to investigate sparser graphs.  

The rest of the paper is devoted to prove Theorem~\ref{thm:main}.

\section{Proof of Theorem~\ref{thm:main}}

\subsection{Typical properties of $\G(n,p)$ and useful inequalities}

Let us start by listing some typical properties of $\G(n,p)$. These observations are part of folklore and can be found in many places, so we will usually skip proofs, pointing to corresponding results in existing literature. Let $N_i(v)$ denote the set of vertices at distance $i$ from $v$, and let $N_i[v]$ denote the set of vertices within distance $i$ of $v$, that is, $N_i[v] = \bigcup_{0 \le j \le i} N_j(v)$. For simplicity, we use $N[v]$ to denote $N_1[v]$, and $N(v)$ to denote $N_1(v)$. Since cops occupy edges but the robber occupies vertices, we will need to investigate the set of edges at ``distance'' $i$ from a given vertex $v$ that we denote by $E_i(v)$. Formally, $E_i(v)$ consists of edges between $N_{i-1}(v)$ and $N_i(v)$, and within $N_{i-1}(v)$. In particular, $E_1(v)$ is the set of edges incident to $v$. Finally, let $P_i(v,w)$ denote the number of paths of length $i$ joining $v$ and $w$.

\bigskip 

Let us start with the following lemma. 

\begin{lemma}\label{lem:elem1}
Let $d=d(n) = p(n-1) \ge \log^3 n$. Then, there exists a positive constant $c$ such that a.a.s.\ the following properties hold in $\G(n,p) = (V,E)$.
\begin{enumerate}
\item Let $S \subseteq V$ be any set of $s=|S|$ vertices, and let $r \in \N$. Then
$$
\left| \bigcup_{v \in S} N_r[v] \right| \ge c \min\{s d^r, n \}.
$$
Moreover, if $s$ and $r$ are such that $s d^r < n / \log n$, then
$$
\left| \bigcup_{v \in S} N_r[v] \right| = (1+o(1)) s d^r.
$$
\item $\G(n,p)$ is connected.
\item Let $r = r(n)$ be the largest integer such that $d^r \le \sqrt{n \log n}$. Then, for every vertex $v \in V$ and $w \in N_{r+1}(v)$, the number of edges from $w$ to $N_r(v)$ is at most $b$, where 
$$
b = 
\begin{cases}
250 & \text{ if } d \le n^{0.49} \\
\frac {3 \log n}{\log \log n} & \text{ if } n^{0.49} < d \le \sqrt{n}.
\end{cases}
$$
\end{enumerate}
\end{lemma}
\begin{proof}
The proof of part (i) can be found in~\cite{PW_gnp}. The fact that $\G(n,p)$ is connected is well known (see, for example,~\cite{JLR}). In fact, the (sharp) threshold for connectivity is $p = \log n / n$ so this property holds for even sparser graphs. 

For part (iii), let us first expose the $r$th neighbourhood of $v$.  By part (i), we may assume that $|N_r[v]|=(1+o(1)) d^{r} < 2 d^r$. For any $w \in V \setminus N_r[v]$, the probability that there are at least $b$ edges joining $w$ to $N_r(v)$ is at most
$$
q := {2 d^r \choose b} p^b \le \left( \frac{2ed^r}{b} \right)^b \left( \frac dn \right)^b = \left( \frac{2ed^{r+1}}{bn} \right)^b.
$$
If $d \le n^{0.49}$, then
$$
q \le \left( \frac{2e d \sqrt{n \log n}}{bn} \right)^b \le n^{-0.005 b} = o(n^{-2}),
$$
provided that $b$ is large enough (say, $b = 250$). For $n^{0.49} < d \le \sqrt{n}$ (and so $r=1$), we observe that
$$
q \le \left( \frac{2e}{b} \right)^b = \exp \left( - (1+o(1)) b \log b \right) = o(n^{-2}),
$$
provided $b = 3 \log n / \log \log n$. The claim follows by the union bound over all pairs $v, w$. The proof of the lemma is finished.
\end{proof}

\bigskip 

The next lemma can be found in~\cite{lp2}. (See also~\cite{lazy_gnp} for its extension.)

\begin{lemma}\label{lem:elem2}
Let $\eps$ and $\alpha$ be constants such that $0<\eps<0.1$, $\eps<\alpha<1-\eps$, and let $d=d(n)=p(n-1)=n^{\alpha+o(1)}$. Let $\ell \in \N$ be the largest integer such that $\ell < 1/\alpha$. Then, a.a.s.\ for every vertex $v$ of $\G(n,p)$ the following properties hold.
\begin{enumerate}
\item [(i)] If $w\in N_i[v]$ for some $i$ with $2\le i \le \ell$, then $P_i(v,w) \le \frac {3}{1-i\alpha}$.
\item [(ii)] If $w\in N_{\ell+1}[v]$ and $d^{\ell+1} \ge 7 n \log n$, then $P_{\ell+1}(v,w) \le \frac{6}{1-\ell \alpha}\frac{d^{\ell+1}}{n}$.
\item [(iii)] If $w\in N_{\ell+1}[v]$ and $d^{\ell+1} < 7 n \log n$, then $P_{\ell+1}(v,w) \le \frac{42}{1-\ell \alpha} \log n$.
\end{enumerate}
Moreover, a.a.s.\
\begin{enumerate}
\item [(iv)] Every edge of $\G(n,p)$ is contained in at most $\eps d$ cycles of length at most $\ell+2$.
\end{enumerate}
\end{lemma}

\bigskip

We will also use the following variant of Chernoff's bound (see, for example,~\cite{JLR}):
\begin{lemma}[\textbf{Chernoff Bound}]
If $X$ is a binomial random variable with expectation $\mu$, and $0<\delta<1$, then 
$$
\Pr[X < (1-\delta)\mu] \le \exp \left( -\frac{\delta^2 \mu}{2} \right)
$$ 
and if $\delta > 0$,
$$
\Pr [X > (1+\delta)\mu] \le \exp \left(-\frac{\delta^2 \mu}{2+\delta} \right).
$$
\end{lemma}

\subsection{Upper bound}

First, let us deal with dense graphs that correspond to part (iii) of Theorem~\ref{thm:main}. In fact, we are going to make a simple observation that the containability number is linear if $G$ has a perfect or a near-perfect matching. The result will follow since it is well-known that for $p=p(n)$ such that $pn - \log n \to \infty$, $\G(n,p)$ has a perfect (or a near-perfect) matching a.a.s. (As usual, see~\cite{JLR}, for more details.)

\begin{lemma}\label{lem:perfect}
Suppose that $G$ on $n$ vertices has a perfect matching ($n$ is even) or a near-perfect matching ($n$ is odd). Then, $\xi(G) \le n$.
\end{lemma}
\begin{proof}
Suppose first that $n$ is even. The cops start on the edges of a perfect matching; two cops occupy any edge of the matching for a total of $n$ cops. All vertices of $G$ can be associated with unique cops. The robber starts on some vertex $v$. One edge incident to $v$ (the edge $vv'$ that belongs to the perfect matching used) is already occupied by a cop (in fact, by two cops, associated with $v$ and $v'$). Moreover, the remaining cops can move so that all edges incident to $v$ are protected and the game ends. Indeed, for each edge $vu$, the cop associated with $u$ moves to $vu$. 

The case when $n$ is odd is also very easy. Two cops start on each edge of a near-perfect matching which matches all vertices but $u$. If $u$ is isolated, we may simply remove it from $G$ and arrive back to the case when $n$ is even. (Recall that the cops win if all edges incident with the robber are occupied by cops. As this property is vacuously true when the robber starts on an isolated vertex, we may assume that she does not start on $u$.) Hence, we may assume that $u$ is not isolated. We introduce one more cop on some edge incident to $u$. The total number of cops is at most $2 \cdot \frac {n-1}{2} + 1 = n$; again, each vertex of $G$ can be associated with a unique cop and the proof goes as before.
\end{proof}

\bigskip

Now, let us move to the following lemma that yields part (i) of Theorem~\ref{thm:main}. We combine and adjust ideas from both~\cite{lp2} and~\cite{PW_gnp} in order to include much sparser graphs.  Cases when $\alpha = 1/k$ for some positive integer $k$ are also covered. 

\begin{lemma}
Let $d=d(n) = p(n-1) \ge \log^3 n$. Suppose that there exists a positive integer $r=r(n)$ such that 
$$
(n \log n)^{\frac {1}{2r+1}} \le d \le (n \log n)^{\frac {1}{2r}}.
$$
Then, a.a.s.\ 
$$
\xi(\G(n,p)) = O(d^{r+1}).
$$
\end{lemma}

\begin{proof}
Since our aim is to prove that the desired bound holds a.a.s.\ for $\G(n,p)$, we may assume, without loss of generality, that a graph $G$ the players play on satisfies the properties stated in Lemma~\ref{lem:elem1}. A team of cops is determined by independently choosing each edge of $e \in E(G)$ to be occupied by a cop with probability $Cd^r/n$, where $C$ is a (large) constant to be determined soon. It follows from Lemma~\ref{lem:elem1}(i) that $G$ has $(1+o(1)) dn/2$ edges. Hence, the expected number of cops is equal to
$$
(1+o(1)) \frac {dn}{2} \cdot \frac {Cd^r}{n} = (1+o(1)) \frac {Cd^{r+1}}{2} .
$$
It follows from Chernoff's bound that the total number of cops is $\Theta(d^{r+1})$ a.a.s.

The robber appears at some vertex $v \in V(G)$. Let $X \subseteq E(G)$ be the set of edges between $N_r(v)$ and $N_{r+1}(v)$. It follows from Lemma~\ref{lem:elem1}(i) that 
$$
|X| \le (1+o(1)) d |N_r(v)| \le 2 d^{r+1}.
$$
Our goal is to show that with probability $1-o(n^{-1})$ it is possible to assign distinct cops to all edges $e$ in $X$ such that a cop assigned to $e$ is within distance $(r+1)$ of $e$. (Note that here, the probability refers to the randomness in distributing the cops; the graph $G$ is fixed.) If this can be done, then after the robber appears these cops can begin moving straight to their assigned destinations in $X$. Since the first move belongs to the cops, they have $(r+1)$ steps, after which the robber must still be inside $N_r[v]$, which is fully occupied by cops. She is ``trapped'' inside $N_r[v]$, so we can send an auxiliary team of, say, $2 d^{r+1}$ cops to go to every edge in the graph induced by $N_r[v]$, and the game ends. Hence, the cops will win with probability $1-o(n^{-1})$, for each possible starting vertex $v \in V(G)$. It will follow that the strategy gives a win for the cops a.a.s.  

Let $Y$ be the (random) set of edges occupied by cops. Instead of showing that the desired assignment between $X$ and $Y$ exists, we will show that it is possible to assign $b(u)$ distinct cops to all vertices $u$ of $N_{r+1}(v)$, where $b(u)$ is the number of neighbours of $u$ that are in $N_r(v)$ (that is, the number of edges of $X$ incident to $u$) and such that each cop assigned to $u$ is within distance $(r+1)$ from $u$. 
(Note that this time ``distance'' is measured between vertex $u$ and edges which is non-standard. In this paper, we define it as follows: edge $e$ is at distance at most $(r+1)$ from $u$ if $e$ is at distance at most $r$ from some edge adjacent to $u$.)
Indeed, if this can be done, assigned cops run to $u$, after $r$ rounds they are incident to $u$, and then spread to edges between $u$ and $N_r(v)$; the entire $X$ is occupied by cops. In order to show that the required assignment between $N_{r+1}(v)$ and $Y$ exists with probability $1-o(n^{-1})$, we show that with this probability, $N_{r+1}(v)$ satisfies Hall's condition for matchings in bipartite graphs. 

Suppose first that $d \le n^{0.49}$ and fix $b=250$. It follows from Lemma~\ref{lem:elem1}(iii) that $b(u) \le b$ for every $u \in N_{r+1}(v)$. Set 
$$
k_0 = \max \{ k : k d^{r} < n \}.
$$ 
Let $K \subseteq N_{r+1}(v)$ with $|K|=k \le k_0$. We may apply Lemma~\ref{lem:elem1}(i) to bound the size of $\bigcup_{u \in K} N_r[u]$ and the number of edges incident to each vertex. It follows that the number of edges of $Y$ that are incident to some vertex in $\bigcup_{u \in K} N_r[u]$ can be stochastically bounded from below by the binomial random variable ${\rm Bin}(\lfloor c k d^r \cdot (d/3) \rfloor, Cd^r/n)$, whose expected value is asymptotic to $(Cc/3) k d^{2r+1} / n \ge (Cc/3) k \log n$. Using Chernoff's bound we get that the probability that there are fewer than $bk$ edges of $Y$ incident to this set of vertices is less than $\exp(-4k \log n)$ when $C$ is a sufficiently large constant. Hence, the probability that the sufficient condition in the statement of Hall's theorem fails for at least one set $K$ with $|K|\le k_0$ is at most
$$
\sum_{k=1}^{k_0} {|N_{r+1}(v)| \choose k} \exp( - 4 k \log n)  \le \sum_{k=1}^{k_0} n^k \exp( - 4 k \log n) = o(n^{-1}).
$$

Now consider any set  $K \subseteq N_{r+1}(v)$ with $k_0 < |K| = k \le |N_{r+1}(v)| \le 2 d^{r+1}$ (if such a set exists). Lemma~\ref{lem:elem1}(i) implies that the size of $\bigcup_{u \in K} N_r[u]$ is at least  $cn$, so we expect at least $cn \cdot (d/3) \cdot Cd^r/n = (Cc/3) d^{r+1}$ edges of $Y$ incident to this set.  Again using Chernoff's bound, we deduce that the number of edges of $Y$ incident to this set is at least $2 b d^{r+1} \ge b |N_{r+1}(v)| \ge bk$ with probability at least $1-\exp(- 4 d^{r+1})$, by taking the constant $C$ to be large enough. Since
$$
\sum_{k=k_0+1}^{|N_{r+1}(v)|} {|N_{r+1}(v)| \choose k} \exp( - 4 d^{r+1} ) \le 2^{2 d^{r+1}} \exp( - 4 d^{r+1} ) = o(n^{-1}),
$$
the necessary condition in Hall's theorem holds with probability $1 - o(n^{-1})$. 

Finally, suppose that $d > n^{0.49}$. Since Lemma~\ref{lem:perfect} implies that the result holds for $d > \sqrt{n}$, we may assume that $d \le \sqrt{n}$. (In fact, for $d > \sqrt{n}$ we get a better bound of $n$ rather than $O(d^2)$ that we aim for.) This time, set $b = 3 \log \log n / \log n$ to make sure $b(u) \le b$ for all $u \in N_{r+1}(v)$. The proof is almost the same as before. For small sets of size at most $k_0 = \Theta(n/d)$, we expect $(Cc/3) k d^{3} / n \ge (Cc/3) k n^{0.47}$ edges, much more than we actually need, namely, $b k$. For large sets of size more than $k_0$, we modify the argument slightly and instead of assigning $b$ cops to each vertex of $N_{r+1}(v)$, we notice that the number of cops needed to assign is equal to $\sum_{u \in K} b(u) \le |X| \le 2 d^{r+1}$. (There might be some vertices of $N_{r+1}(v)$ that are incident to $b$ edges of $X$ but the total number of incident edges to $K$ is clearly at most $|X|$.) The rest is not affected and the proof is finished.
\end{proof}

\bigskip

The next lemma takes care of part (ii) of Theorem~\ref{thm:main}.

\begin{lemma}
Let $d=d(n) = p(n-1) \ge \log^3 n$. Suppose that there exists an integer $r=r(n) \ge 2$ such that 
$$
(n \log n)^{\frac {1}{2r}} \le d \le (n \log n)^{\frac {1}{2r-1}}.
$$
Then, a.a.s.\ 
$$
\xi(\G(n,p)) = O \left( \frac {n \log n}{d^{r-1}} \right).
$$
\end{lemma}
\begin{proof}
We mimic the proof of the previous lemma so we skip details focusing only on differences. A team of cops is determined by independently choosing each edge of $e \in E(G)$ to be occupied by a cop with probability $C \log n / d^r$, for the total number of cops $\Theta(n \log n / d^{r-1})$ a.a.s.

The robber appears at some vertex $v \in V(G)$. This time, $X \subseteq E(G)$ is the set of edges between $N_{r-1}(v)$ and $N_{r}(v)$ and $|X| \le 2 d^r$. We show that it is possible to assign $b=250$ distinct cops to all vertices $u$ of $N_r(v)$ such that a cop assigned to $u$ is within ``distance'' $r$ from $u$. The definition of $k_0$ has to be adjusted. Set 
$$
k_0 = \max \{ k : k d^{r-1} < n \}.
$$ 
Let $K \subseteq N_{r}(v)$ with $|K|=k \le k_0$. The expected number of edges of $Y$ that are incident to some vertex in $\bigcup_{u \in K} N_{r-1}[u]$ is at least $(ckd^{r-1}) (d/3) (C \log n / d^r) = (Cc/3) k \log n$, and the rest of the argument is not affected. Now consider any set  $K \subseteq N_{r}(v)$ with $k_0 < |K| = k \le |N_{r}(v)| \le 2 d^{r}$ (if such a set exists). The size of $\bigcup_{u \in K} N_{r-1}[u]$ is at least  $cn$, so we expect at least 
$$
(cn) \left(\frac {d}{3} \right) \left( \frac {C \log n}{d^r} \right) = \frac {Cc d^r n \log n}{3 d^{2r-1}} \ge \frac {Cc d^r }{3} 
$$ 
edges of $Y$ incident to this set. Hence, the number of edges of $Y$ incident to this set is at least $2 b d^{r} \ge b |N_{r}(v)| \ge bk$ with probability at least $1-\exp(- 4 d^{r})$, by taking the constant $C$ to be large enough. The argument we had before works again, and the proof is finished.
\end{proof}

\subsection{Lower bound}

The proof of the lower bound is an adaptation of the proof used for the classic cop number in~\cite{lp2}. The two bounds, corresponding to parts (i) and (ii) in Theorem~\ref{thm:main}, are proved independently in the following two lemmas. 

\begin{lemma}\label{lemma:lowerbound1}
Let $\frac{1}{2j+1}<\alpha<\frac{1}{2j}$ for some integer $j\ge 1$, $c=c(j,\alpha)=\frac{3}{1-2j \alpha}$, and $d=d(n)=np=n^{\alpha+o(1)}$. Then, a.a.s.\
$$
\xi(\G(n,p)) > K := \left( \frac{d}{30c(2j+1)} \right)^{j+1}\,.
$$
\end{lemma}

\begin{proof}
Since our aim is to prove that the desired bound holds a.a.s.\ for $\G(n,p)$, we may assume, without loss of generality, that a graph $G$ the players play on satisfies the properties stated in Lemmas~\ref{lem:elem1} and~\ref{lem:elem2}. Suppose that the robber is chased by $K$ cops. Our goal is to provide a winning strategy for the robber on $G$.  For vertices $x_1,x_2,\dots,x_s$, let $\pl^{x_1,x_2,\dots,x_s}_i(v)$ denote the number of cops in $E_i(v)$ (that is, at distance $i$ from $v$) in the graph $G \setminus \{x_1,x_2, \dots,x_s\}$.

Right before the robber makes her move, we say that the vertex $v$ occupied by the robber is \emph{safe}, if for some neighbour $x$ of $v$ we have $\pl^x_1(v) \le\frac{d}{30c(2j+1)}$, and
$$
\pl_{2i}^x(v), \pl_{2i+1}^x(v)\le \left( \frac{d}{30c(2j+1)} \right)^{i+1}
$$
for $i=1,2,\dots,j-1$ (such a vertex $x$ will be called a \emph{deadly neighbour} of $v$). The reason for introducing deadly neighbours is to deal with a situation that many cops apply a greedy strategy and always decrease the distance between them and the robber. As a result, there might be many cops ``right behind'' the robber but they are not so dangerous unless she makes a step ``backwards'' by moving to a vertex she came from in the previous round, a deadly neighbour! Moreover, note that a vertex is called safe for a reason: if the robber occupies a safe vertex, then the game is definitely not over since the condition for $C_1^x(v)$ guarantees that at most a small fraction of incident edges are occupied by cops.

Since a.a.s.\ $G$ is connected (see Lemma~\ref{lem:elem1}(ii)), without loss of generality we may assume that at the beginning of the game all cops begin at the same edge, $e$. Subsequently, the robber may choose a vertex $v$ so that $e$ is at distance $2j+2$ from $v$ (see Lemma~\ref{lem:elem1}(i) applied with $r=2j+1$ to see that almost all vertices are at distance $2j+1$ from both endpoints of $e$). Hence, even if all cops will move from $e$ to $E_{2j+1}(v)$ after this move, $v$ will remain safe as no bound is required for $\pl_{2j+1}^x(v)$. (Of course, again, without loss of generality we may assume that all cops pass for the next round and stay at $e$ before starting applying their best strategy against the robber.) Hence, in order to prove the lemma, it is enough to show that if the robber's current vertex $v$ is safe, then she can move along an unoccupied edge to a neighbour $y$ so that no matter how the cops move in the next round, $y$ remains safe.

For $0 \le r \le 2j$, we say that a neighbour $y$ of $v$ is {\em $r$-dangerous} if
\begin{itemize}
\item [(i)] an edge $vy$ is occupied by a cop (for $r=0$)\,, or
\item [(ii)] $\pl^{v,x}_r(y) \ge \frac 13 \left( \frac{d}{30c(2j+1)} \right)^{i}$ (for $r=2i$ or $r=2i-1$, where $i = 1,2, \ldots, j$)\,,
\end{itemize}
where $x$ is a deadly neighbour of $v$. We will check that for every $r \in \{0, 1, \ldots, 2j\}$, the number of $r$-dangerous neighbours of $v$, which we denote by $\dang(r)$, is smaller than $\frac {d}{2(2j+1)}$. Clearly, since $v$ is safe, 
$$
\dang(0) \le \pl^x_1(v) \le \frac{d}{30c(2j+1)} \le \frac {d}{2(2j+1)}.
$$
Suppose then that $r=2i$ or $r=2i-1$ for some $i \in \{1,2,\ldots, j\}$. Every $r$-dangerous neighbour of $v$ has at least $\frac 13 \left( \frac{d}{30c(2j+1)} \right)^i$ cops occupying $E_{\le (r+1)}(v)$. On the other hand, every edge from $E_{\le (r+1)}(v)$ is incident to at most $2$ vertices at distance at most $r$ from $v$. Moreover, Lemma~\ref{lem:elem2}(i) implies that there are at most $c$ paths between $v$ and any $w \in N_{\le r}(v)$. Finally, by the assumption that $v$ is safe, we have $\pl^{x}_{2i}(v), \pl^{x}_{2i+1}(v) \le  \left( \frac{d}{30c(2j+1)} \right)^{i+1}$, provided that $i \le j-1$; the corresponding conditions for $\pl_{2j}^x(v)$ and $\pl_{2j+1}^x(v)$ are trivially true, since both can be bounded from above by $K$, the total number of cops. Combining all of these yields
\begin{eqnarray*}
\frac 13 \left( \frac{d}{30c(2j+1)} \right)^i \cdot \dang(r) &\le& 2c \cdot \pl_{\le (r+1)}^x(v) \le 2c \cdot (2+o(1)) \pl_{r+1}^x(v) \\
&\le& 5 c \cdot \left( \frac{d}{30c(2j+1)} \right)^{i+1},
\end{eqnarray*}
and consequently $\dang(r) \le \frac {d}{2(2j+1)}$, as required. Thus, there at most $d/2$ of neighbours of $v$ are $r$-dangerous for some $r \in \{0,1,\dots,2j\}$. 

Since we have $(1+o(1)) d$ neighbours to choose from (see Lemma~\ref{lem:elem1}(i)), there are plenty of neighbours of $v$ which are not $r$-dangerous for any $r=0,1,\dots,2j$ and the robber might want to move to one of them. However, there is one small issue we have to deal with. In the definition of being dangerous, we consider the graph $G \setminus \{v,x\}$ whereas in the definition of being safe we want to use $G \setminus \{v\}$ instead. Fortunately, Lemma~\ref{lem:elem2}(iv) implies that we can find a neighbour $y$ of $v$ that is not only not dangerous but also $x$ does not belong to the $2j$-neighbourhood of $y$ in $G\setminus \{v\}$. It follows that $vy$ is not occupied by a cop and $\pl^{v}_r(y) < \frac 13 \left( \frac{d}{30c(2j+1)} \right)^{i}$ for $r=2i$ or $r=2i-1$, where $i = 1,2, \ldots, j$. We move the robber to $y$.

Now, it is time for the cops to make their move. Because of our choice of the vertex $y$, we can assure that the desired upper bound for $\pl_r^v(y)$ required for $y$ to be safe will hold for $r \in \{1,2,\dots,2j-1\}$. Indeed, the best that the cops can do to try to fail the condition for $\pl_r^v(y)$ is to move all cops at distance $r-1$ and $r+1$ from $y$ to $r$-neighbourhood of $y$, and to make cops at distance $r$ stay put, but this would not be enough. Thus, regardless of the strategy used by the cops, $y$ is safe and the proof is finished.
\end{proof}

\begin{lemma}\label{lemma:lowerbound2}
Let $\frac{1}{2j}<\alpha<\frac{1}{2j-1}$ for some integer $j\ge 1$,  $\bc=\bc(\alpha)=\frac{6}{1-(2j-1) \alpha}$ and $d=d(n)=np=n^{\alpha+o(1)}$. Then, a.a.s.\
$$
\xi(\G(n,p))\ge \bar{K} := \left( \frac{d}{30 \bc (2j+1)} \right)^{j+1} \frac{n}{d^{2j}}\,.
$$
\end{lemma}

\begin{proof}
The proof is very similar to that of Lemma~\ref{lemma:lowerbound1}. The only difference is that checking the desired bounds for $\dang(2j-1)$ and $\dang(2j)$ is slightly more complicated. As before, we do not control the number of cops in $E_{2j}(v)$ and $E_{2j+1}(v)$, clearly $\pl^x_{2j}(v)$ and $\pl^x_{2j+1}(v)$ are bounded from above by $\bar{K}$, the total number of cops. We get
\begin{eqnarray*}
\frac 13 \left( \frac{d}{30\bc(2j+1)} \right)^j \cdot \dang(2j-1) &\le& 2\bc \cdot \pl_{\le (2j)}^x(v) \le 2\bc \cdot (2+o(1)) \bar{K} \\
&\le& 5\bc \cdot \left( \frac{d}{30\bc(2j+1)} \right)^{j+1},
\end{eqnarray*}
and consequently $\dang(2j-1) \le \frac {d}{2(2j+1)}$, as required. (Note that we have room to spare here but we cannot take advantage of it so we do not modify the definition of being $(2j-1)$-dangerous.)

Let us now notice that a cop at distance $2j+1$ from $v$ can contribute to the ``dangerousness'' of more than $\bc$ neighbours of $v$. However, the number of paths of length $2j$ joining $v$ and $w$ is bounded from above by $\bc d^{2j}/n$ (see Lemma~\ref{lem:elem2}(ii) and note that $d^{2j} = n^{2j \alpha + o(1)} \ge 7 n \log n$, since $2j \alpha > 1$). Hence,
\begin{eqnarray*}
\frac 13 \left( \frac{d}{30\bc(2j+1)} \right)^j \cdot \dang(2j) &\le& \frac {2 \bc d^{2j}}{n} \cdot \pl_{\le (2j+1)}^x(v) \le \frac {2 \bc d^{2j}}{n} \cdot (2+o(1)) \bar{K} \\
&\le& 5 \bc \cdot \left( \frac{d}{30\bc(2j+1)} \right)^{j+1},
\end{eqnarray*}
and, as desired, $\dang(2j)\le \frac{d}{2(2j+1)}$. Besides this modification the argument remains basically the same.
\end{proof}

\end{document}